\documentclass[12pt]{amsart}
\usepackage[utf8]{inputenc}
\usepackage{mathrsfs} 
\usepackage{amssymb}
\usepackage{bm}
\usepackage{graphicx}
\usepackage[centertags]{amsmath}
\usepackage{amsfonts}
\usepackage{amsthm}
\usepackage{enumerate}
\usepackage{tocvsec2}
\usepackage{xcolor}
\usepackage[margin=1in]{geometry}

\newtheorem{theorem}{Theorem}[section]
\newtheorem*{theorem*}{Theorem}

\newtheorem{corollary}[theorem]{Corollary}
\newtheorem{lemma}[theorem]{Lemma}
\newtheorem{rem}[theorem]{Remark}

\newtheorem{proposition}[theorem]{Proposition}

\newtheorem*{ta}{Theorem A}
\newtheorem*{tb}{Theorem B}
\newtheorem*{tc}{Theorem C}
\newtheorem*{claim}{Claim}
\theoremstyle{definition}
\newtheorem{definition}[theorem]{Definition}
\newenvironment{remark}[1][Remark]{\begin{trivlist}
\item[\hskip \labelsep {\bfseries #1}]}{\end{trivlist}}

\newcommand{\nn}{\mathbb{N}}

\newcommand{\ee}{\varepsilon}

\newcommand{\aaa}{\mathcal{A}}

\newcommand{\ttt}{\mathcal{T}}
\newcommand{\sss}{\mathcal{S}}

\newcommand{\supp}{\mathrm{supp}}

\newcommand{\uuu}{\mathcal{U}}

\newcommand{\ord}{\textbf{Ord}}
\newcommand{\cat}{^\smallfrown}

\begin{document}
\title[Quantitative Factorization]{Quantitative Factorization of weakly compact, \\  Rosenthal, and $\xi$-Banach-Saks operators}
\author{K. Beanland}
\address{Department of Mathematics, Washington and Lee University, Lexington, VA 24450}
\email{beanlandk@wlu.edu}
\author{R.M. Causey}
\address{Department of Mathematics, Miami University, Oxford, OH 45056}
\email{causeyrm@miamioh.edu}

\begin{abstract} We prove quantitative factorization results for several classes of operators, including weakly compact, Rosenthal, and $\xi$-Banach-Saks operators.  

\end{abstract}
\maketitle

\section{Introduction}

In recent literature \cite{BCFW,Brooker}, ordinal indices are used to define several new classes of operators. The main results of the present paper are factorization results for these new classes analogous to the celebrated Davis, Figiel, Johnson  and Pe\l czy\'{n}ski factorization theorem for weakly compact operators. Before we state these results, we recall some of these new classes. In the following for an operator $A$, $Sz(A)$ is the Szlenk index and $\mathcal{J}(A)$ is the James index of the operators (the James index was defined in \cite{Cweakly}). We recall the necessary definitions in a subsequent section.  Let $\textbf{Ord}$ be the class of ordinal numbers and $\xi\in \textbf{Ord}$. 
\begin{enumerate}
\item Let $\mathscr{SZ}_\xi$ denote the class of all operators $A$ so that $Sz(A) \leqslant \omega^\xi$. The class  $\cup_{\xi \in \textbf{Ord}} \mathscr{SZ}_\xi$ is the class of Asplund operators.
\item Let $\mathfrak{NP}^\xi_1$ denote the class of all operators that do not preserve Bourgain $\ell_1$ trees of order $\omega^\xi$. The class $\cup_{\xi \in \textbf{Ord}} \mathfrak{NP}^\xi_1$ consists of the operators which do not preserve a copy of $\ell_1$ (also known as \emph{Rosenthal operators}).
\item Let $\mathscr{J}_{\xi}$ denote the class of operators $A$ so that $\mathcal{J}(A) \leqslant \xi$. The class $\cup_{\xi \in \textbf{Ord}} \mathscr{J}_\xi=:\mathscr{J}$ is the class of weakly compact operators.
\item For $0<\xi<\omega_1$ let $\mathfrak{SM}_1^\xi$ denote the class of operators that do not preserve an $\ell_1^\xi$ spreading model. 
\item For $0<\xi<\omega_1$ let $\mathfrak{WC}^\xi$ denote the class consisting of those weakly compact operators which lie in $\mathfrak{SM}_1^\xi$. 
\end{enumerate}

In \cite{BCFW,Brooker} it is shown that each of the above classes in (1), (2), (4), and (5) are distinct for different ordinals. The classes $\mathscr{J}_{\omega}$ and $\mathfrak{NP}^1_1$ correspond to the ideals of super weakly compact and super Rosenthal operators, respectively (\cite{Cweakly}, \cite{BCFW}). The following result from \cite{BCFW,Brooker} states that for certain ordinals, the above subclasses are in fact closed two-sided operator ideals. 

\begin{theorem}
Let $\xi \in \textbf{\emph{Ord}}$. The classes $\mathscr{SZ}_\xi$, $\mathfrak{NP}^{\omega^\xi}_1$ and $\mathscr{J}_{\omega^{\omega^\xi}}$ are closed, two-sided ideals. Moreover if $0<\xi <\omega_1$ the classes $\mathfrak{SM}_1^\xi$ and $\mathfrak{WC}^\xi$ are closed, two-sided ideals.
\end{theorem}

 In general if  $\mathscr{I}$ is an operator ideal, we let $\mbox{Space}(\mathscr{I})$ denote the collection of Banach spaces $X$ so that the identity $I_X$ lies in $\mathscr{I}$.  We say that an operator ideal $\mathscr{I}$  has the {\em factorization property} if for every $A\in \mathscr{I}$ there is an $X \in \mbox{Space}(\mathscr{I})$ so that $A$ factors through $X$. The famous theorem of Davis, Figiel, Johnson and Pe\l czy\'{n}ski \cite{DFJP} states that the class of  weakly compact operators has the factorization property. It is known that the classes of Rosenthal and Banach-Saks operators also possess the factorization property (see \cite{Heinrich} and references therein). The current paper is concerned with results of this kind as they correspond to the new classes of operators defined above. We first note that there are well studied operator ideals that do not possess the factorization property. We will give an example in Section $3$ which proves the following proposition. 
 

\begin{proposition}
Neither the class of  super weakly compact operators nor the class of super Rosenthal operators possesses the factorization property. \label{superstuff}
\end{proposition}

 More generally, if $\mathscr{I}$, $\mathscr{M}$ are two classes of operators, we may say $\mathscr{I}$ has the $\mathscr{M}$ \emph{factorization property} if every member of $\mathscr{I}$ factors through a member of $\text{Space}(\mathscr{M})$. The starting point for our quantitative factorization results is the work of Brooker \cite{Brooker}, who showed that for every ordinal $\xi$, $\in \mathscr{SZ}_\xi$ has the $\mathscr{SZ}_{\xi+1}$ factorization property. Moreover, Brooker showed that both classes $$\{\xi\in \ord: \mathscr{SZ}_\xi \text{\ has the factorization property}\},$$ $$\{\xi\in \ord: \mathscr{SZ}_\xi\text{\ does not have the factorization property}\}$$ are proper classes (that is, unbounded classes) of ordinals.  The question of determining exactly those $\xi$ such that $\mathscr{SZ}_\xi$ possesses the factorization property is still open.    The main results of the current paper are analogous to Brooker's result for the different operator ideals listed above. As in the proof of Brooker's result mentioned above, our main tool is a theorem of Heinrich \cite{Heinrich} which yields factorization results for  $\Sigma_p$ pairs of classes of operators.

\begin{ta}
For $\xi \in \textbf{\emph{Ord}}$ with $0< \xi< \omega_1$,  $\mathfrak{SM}_1^\xi$ and $ \mathfrak{WC}^\xi$ have the factorization property.
\end{ta}

\begin{tb}
Let $\xi \in \textbf{\emph{Ord}}$. The following hold:
\begin{itemize}
\item[(i)] The class $\mathfrak{NP}^{\omega^\xi}_1$ has the $\mathfrak{NP}^{\omega^\xi+1}_1$ factorization property. Moreover $\mathfrak{NP}^{\omega^\xi}_1$ has the factorization property if and only if $\xi$ has uncountable cofinality.
\item[(ii)] The class  $\mathscr{J}_{\omega^{\omega^\xi}}$ has the $\mathscr{J}_{\omega^{\omega^{\xi+1}}}$ factorization property.  
\end{itemize}
\end{tb}

These results together with some deep descriptive set theoretic results from \cite{Dodos} and \cite{DF} yield the following.  In what follows, $\mathfrak{X}$ denotes the class of operators with separable range.

\begin{tc} For each countable ordinal $\xi$, there exists a separable Banach space $S$ which is reflexive (respectively, which contains no copy of $\ell_1$) such that every member of $\mathscr{J}_\xi\cap \mathfrak{X}$ (respectively, $\mathfrak{NP}_1^\xi\cap \mathfrak{X}$) factors through a subspace (respectively, quotient) of $S$.  

In particular, there exist separable Banach spaces $S_1, S_2$ such that $S_1$ is reflexive, $S_2$ contains no copy of $\ell_1$, every super weakly compact operator with separable range factors through a subspace of $S_1$, and every super Rosenthal operator with separable range factors through a quotient of $S_2$.

\end{tc}

We note that in \cite{F}, Figiel showed that there exists a separable, reflexive Banach space $Z$ such that every compact operator factors through a subspace of $Z$.  Theorem C extends this result, since the class of super weakly compact operators with separable range contains the class of compact operators.  Moreover, Johnson and Szankowski \cite{JS} showed that there does not exist a separable Banach space through which all compact operators factor.  This shows that if $S_1$ is the Banach space from Theorem C, the restriction that every super weakly compact operator with separable range factors only through a subspace of $S_1$, and not through $S_1$ itself, cannot be removed.

\section{Terminology}
\subsection{Classes of operators}
We let $\textbf{Ban}$ denote the class of Banach spaces.  We let $\mathscr{L}$ denote the class of operators between Banach spaces.   For each pair $E,F\in \textbf{Ban}$ of  Banach spaces, $\mathscr{L}(E,F)$ will denote the operators from $E$ into $F$.  Given a  class $\mathscr{M}$ of operators, $\mathscr{M}(E,F)=\mathscr{M}\cap \mathscr{L}(E,F)$.  Recall that $\mathscr{M}$ is said to have \emph{the ideal property} if for every $E,F,G,H\in \textbf{Ban}$, $A\in \mathscr{L}(G,H)$, $B\in \mathscr{M}(F,G)$, and $C\in \mathscr{L}(E,F)$, then $ABC\in \mathscr{M}(E,H)$.   We say $\mathscr{M}$ is an \emph{operator ideal} if $\mathscr{M}$ has the ideal property, $I_\mathbb{K}\in \mathscr{M}$, and for every $E,F\in \textbf{Ban}$, $\mathscr{M}(E,F)$ is a vector space.  Here, $I_\mathbb{K}$ is the identity of the scalar field $\mathbb{K}$.  We say $\mathscr{M}$ is \begin{enumerate}[(i)]\item \emph{closed} if for every $E,F\in \textbf{Ban}$, $\mathscr{M}(E,F)$ is a closed subset of $\mathscr{L}(E,F)$ with its norm topology, \item \emph{injective} if for any $E,F,G\in \textbf{Ban}$, any $A\in \mathscr{L}(E,F)$, and any isomorphic embedding $j:F\to G$, if $jA\in \mathscr{M}(E,G)$, then $A\in \mathscr{M}(E,F)$, \item \emph{surjective} if for any $E,F,G\in \textbf{Ban}$, any surjection map $q:G\to E$, and any $A\in \mathscr{L}(E,F)$, if $Aq\in \mathscr{M}(G,F)$, then $A\in \mathscr{M}(E,F)$.  \end{enumerate} 

Given an operator ideal $\mathscr{M}$, the super ideal of $\mathscr{M}$ is the class of those operators $A:X\to Y$ such that for every ultrafilter $\uuu$, the the induced operator $A_\uuu:X_\uuu\to Y_\uuu$ between the ultrapowers lies in $\mathscr{M}$. 

Given a class $\mathscr{M}$ of operators, we let $\text{Space}(\mathscr{M})$ denote the class of Banach spaces $Z$ such that $I_Z\in \mathscr{M}$.   Finally, given two classes of operators $\mathscr{M}$, $\mathscr{I}$ and $1<p<\infty$, we say $(\mathscr{M}, \mathscr{I})$ is a $\Sigma_p$-\emph{pair} if for every pair of sequences of Banach spaces $(X_n:n\in \nn)$, $(Y_n:n\in \nn)$ and every operator $A:(\oplus_n X_n)_{\ell_p}\to (\oplus_n Y_n)_{\ell_p}$ such that $Q_n A P_m\in \mathscr{M}$ for every $m,n\in\nn$, $A\in \mathscr{I}$.   

\begin{theorem}\cite[Theorem $2.1$]{Heinrich} Suppose $\mathscr{M}$, $\mathscr{J}$ are two injective, surjective, closed classes of operators such that $\mathscr{M}$ is an operator ideal and $\mathscr{J}$ possesses the ideal property.  Suppose also that for some $1<p<\infty$, $(\mathscr{M}, \mathscr{J})$ is a $\Sigma_p$ pair.  Then every member of $\mathscr{M}$ factors through a member of $\text{\emph{Space}}(\mathscr{J})$.   
\label{Heinrich}
\end{theorem}

This theorem was not stated in this way in \cite{Heinrich}.  We leave it to the reader to verify that the proof goes through with only notational changes under the hypotheses here (see \cite{Brooker} for further remarks regarding this use of Theorem \ref{Heinrich}).

\subsection{Trees}

Given a set $\Lambda$, we let $\Lambda^{<\nn}$ denote the finite sequences in $\Lambda$, including the empty sequence $\varnothing$.   We order $\Lambda^{<\nn}$ by letting $s\preceq t$ if $s$ is an initial segment of $t$.  We let $s\cat t$ denote the concatenation of $s$ and $t$.  We let $|s|$ denote the length of $s$, and if $0\leqslant i\leqslant |s|$, we let $s|_i$ denote the initial segment of $s$ having length $i$. Given two trees $S,T$, we say a function $\phi:S\to T$ is \emph{monotone} if for any $s\prec s_1\in S$, $\phi(s)\prec \phi(s_1)$.  We let $MAX(T)$ denote those members of $T$ which are maximal with respect to $\prec$.  We let $T'=T\setminus MAX(T)$.  We define the higher order derived trees by transfinite induction by $$T^0=T,$$ $$T^{\xi+1}=(T^\xi)',$$ and $$T^\xi=\underset{\zeta<\xi}{\bigcap}T^\zeta, \text{\ \ }\xi\text{\ is a limit ordinal.}$$  If there exists an ordinal $\xi$ such that $T^\xi=\varnothing$, we let $o(T)$ denote the smallest such $\xi$.  Otherwise we agree to the convention that $o(T)=\infty$. We say $T$ is \emph{well-founded} if $o(T)$ is an ordinal, and say $T$ is \emph{ill-founded} otherwise. We also establish the convention that $\xi<\infty$ for any ordinal $\xi$.   Note that $T$ is ill-founded if and only if there exists an infinite sequence $(\lambda_i)_{i=1}^\infty$ in $\Lambda$ such that for all $n\in \nn$, $(\lambda_i)_{i=1}^n\in T$.  We also define a $B$-\emph{tree}, which is a subset $T$ of some $\Lambda^{<\nn}\setminus\{\varnothing\}$ such that $T\cup \{\varnothing\}$ is a tree. All of the definitions above regarding trees can be relativized to $B$-trees.  

In \cite{Cproximity}, a family $(\ttt_\xi)_{\xi\in \ord}$ of $B$-trees was given such that for any set $\Lambda$, any tree $T$ on $\Lambda$, and any ordinal $\xi$, $o(T)>\xi$ if and only if there exists a collection $(\lambda_t)_{t\in \ttt_\xi}\subset \Lambda$ such that for every $t\in \ttt_\xi$, the sequence $(\lambda_{t|_i})_{i=1}^{|t|}\in T$.  Similarly, if $T$ is a $B$-tree, then $o(T)\geqslant \xi$ if and only if there exists a collection $(\lambda_t)_{t\in \ttt_\xi}$ as above.  

Given a $B$-tree $T$ on $\Lambda$ and $t\in \Lambda^{<\nn}$, we let $T(t)$ denote the non-empty sequences in $\Lambda^{<\nn}$ such that $t\cat s\in T$.  This is also a $B$-tree, and for any ordinal $\xi$, $T^\xi(t)=(T(t))^\xi$.   In particular, if $t\in T$,  $o(T(t))\geqslant \xi$ if and only if $t\in T^\xi$.

\subsection{Schreier families and the repeated averages hierarchy}

We will identify subsets of $\nn$ with strictly increasing sequences in $\nn$ in the natural way.  Therefore the set of finite subsets of $\nn$ can be identified with the subset of $\nn^{<\nn}$ consisting of strictly increasing sequences.  Given finite subsets $E,F$ of $\nn$, we write $E<F$ if $\max E<\min F$ or if either set is empty.  We write $n\leqslant E$ if $n\leqslant \min E$.  

For each $n\in \nn$, we let $$\aaa_n= \bigl\{E\in \nn^{<\nn}: |E|\leqslant n\bigr\}.$$  

We let $\sss_0=\aaa_1$.  If $\sss_\xi$ has been defined, we let $$S_{\xi+1}=\Bigl\{\overset{n}{\underset{i=1}{\bigcup}}E_i: n\in \nn, n\leqslant E_1<\ldots <E_n, \varnothing \neq E_i\in \sss_\xi\Bigr\}.$$  If $\xi$ is a countable limit ordinal and $\sss_\zeta$ has been defined for every $\zeta<\xi$, we fix $\xi_n\uparrow \xi$ and let $$\sss_\xi=\{E: \exists n\leqslant E\in \sss_{\xi_n}\}.$$  Finally, for a countable ordinal $\xi$, $n\in \nn$, and natural numbers $m_1<m_2<\ldots$, with $M=\{m_i\}$, we let $$\sss_\xi[\aaa_n]=\Bigl\{\cup_{i=1}^k E_i:  E_1<\ldots <E_k, (\min E_i)_{i=1}^k\in \sss_\xi, \varnothing \neq E_i\in \aaa_n \Bigr\}$$ and $$\sss_\xi[\aaa_n](M)=\bigl\{\{m_i:i\in E\}: E\in \sss_\xi[\aaa_n]\bigr\}.$$   

In \cite{AMT}, the \emph{repeated averages hierarchy} was defined.  The precise definition of the hierarchy is not necessary for this work, so we only state the properties we will need. Proofs of these facts can be found in \cite{AMT}.   Given a scalar sequence $s=(s_n)$, we let $\supp(s)=\{n\in \nn: s_n\neq 0\}$.  We let $c_{00}$ denote those scalar sequences with finite support and let $(e_i)$ denote the canonical Hamel basis of $c_{00}$.   For every $0\leqslant \xi<\omega_1$ and for every infinite subset $M$ of $\nn$, the sequence $(\xi^M_n)_{n\in \nn}$ is a sequence of members of $c_{00}\cap [0,1]^\nn$ such that \begin{enumerate}[(i)]\item for all $n\in \nn$, $\|\xi^M_n\|_{\ell_1}=1$, \item for all $n\in \nn$, $\supp(\xi^M_n)<\supp(\xi^M_{n+1})$, \item $\cup_{n=1}^\infty \supp(\xi^M_n)=M$. \end{enumerate} We also remark that if $M=(m_n)$ with $m_1<m_2<\ldots$, for each $n\in \nn$, $0^M_n= e_{m_n}$.  Moreover, if $(E_n)$ is the partition of $M$ such that $E_1<E_2<\ldots$ and $|E_n|=\min E_n$, $1^M_n= |E_n|^{-1}\sum_{i\in E_n}e_i$.

We will write $\xi^M_n=(\xi^M_n(i))_{i\in \nn}$.   Given a sequence $s=(x_n)$ in a Banach space, we let $\xi^M.s$ denote the sequence $(y_n)$ where $y_n=\sum_i \xi^M_n(i)x_i$ for each $n\in \nn$.   We say a sequence $s=(x_n)$ is $\xi$-\emph{convergent to} $x$ provided that there exists an infinite subset $N$ of $\nn$ such that for all further infinite subsets $M$ of $N$, the sequence $\xi^M.s$ converges to $x$ in norm.  We say $(x_n)$ is $\xi$-\emph{convergent} if it is $\xi$-convergent to some $x$.  We note that $(x_n)$ $1$-converges to $x$ if and only if it has a subsequence whose Cesaro means converge in norm to $x$, which follows from the description of $(1^M_n)$ in the previous paragraph.

\section{weakly compact operators}

Given an operator $A:X\to Y$ and a constant $\theta>0$, we let $J(A,\theta)$ denote the tree consisting of the empty sequence and those sequences $(x_i)_{i=1}^n\subset B_X$ such that for every $1\leqslant m<n$, every $x\in \text{co}(x_i:i\leqslant m)$, and every $x'\in \text{co}(x_i:m<i\leqslant n)$, $\|Ax-Ax'\|\geqslant \theta$.  We define $\mathcal{J}(A, \theta)=o(J(A, \theta))$ and $\mathcal{J}(A)=\sup_{\theta>0}\mathcal{J}(A, \theta)$.  We collect the following facts from \cite{Cweakly}. 

\begin{theorem}Let $A:X\to Y$ be an operator. \begin{enumerate}[(i)]\item $A$ is weakly compact if and only if $\mathcal{J}(A)<\infty$. \item $A$ is super weakly compact if and only if $\mathcal{J}(A)\leqslant \omega$.  \item the class $\mathscr{J}_{\omega^{\omega^\xi}}$ is a closed operator ideal. \end{enumerate}

\end{theorem}

It is quite obvious that for any ordinal $\xi>0$, $\mathscr{J}_\xi$ is injective and surjective.  Indeed, suppose that $q:E\to F$, $A:F\to G$, and $j:G\to H$ are such that $q$ is a quotient and $j$ is an isometric embedding. For every $x\in F$, choose $e_x\in E$ such that $q e_x=x$ and $\|qe_x\|\leqslant 2\|x\|$.    Then for any $0<\theta_1<\theta_1$, $\{(2^{-1}e_{x_i})_{i=1}^n: (x_i)_{i=1}^n\in J(A, \theta)\}\subset J(jAq, \theta/2).$  Thus $$\mathcal{J}(jAq) \geqslant \mathcal{J}(A),$$ and if $A\notin \mathscr{J}_\xi$, $\mathcal{J}(jAq)>\xi$.

The main result of this section is to prove Theorem B (ii) from the introduction. The following is a restatement of this theorem.  

\begin{theorem} If $A:X\to Y$ is weakly compact and $\mathcal{J}(A)\leqslant \omega^{\omega^\xi}$, then $A$ factors through a member of $\text{\emph{Space}}(\mathscr{J}_{\omega^{\omega^{\xi+1}}})$.  \label{wc}
\end{theorem}

Before passing to the proof of the above theorem we give the proof of Proposition \ref{superstuff} and the first part of Theorem C from the introduction.

\begin{proof}[Proof of Proposition \ref{superstuff}]  For each $n\in \nn$, let $a_n=1/\log(n+1)$ and define $A:c_0\to c_0$ by $A\sum b_ne_n=\sum a_nb_ne_n$. This is a compact (and therefore super weakly compact and super Rosenthal) operator, but for any $n\in \nn$ and $2\leqslant p<\infty$, $$\Bigl(\sum_{i=1}^n \|Ae_i\|^p\Bigr)^{1/p} \geqslant \frac{n^{1/p}}{\log(n+1)}=\frac{n^{1/p}}{\log(n+1)}\mathbb{E}\|\sum_{i=1}^n \ee_ie_i\|,$$ whence $A$ fails to have any non-trivial Rademacher cotype. Thus this operator fails to factor through any Banach space of non-trivial Rademacher cotype, and therefore every Banach space is finitely representable in any Banach space through which $A$ factors.   This gives an example of an operator $A\in \mathscr{J}_\omega$ not factoring through any member of $\text{Space}(\mathscr{J}_\omega)$. It also gives an example of a super Rosenthal operator not factoring through any Banach space in which $\ell_1$ is not finitely representable.  The theorem above yields that every super weakly compact operator factors through a member of $\mathscr{J}_{\omega^\omega}$.  
\end{proof}

 The next theorem is the first part of Theorem C.
Let $\mathfrak{X}$ denote the ideal of operators having separable range. Note that $\mathscr{J}\cap \mathfrak{X}$ is the ideal of operators factoring through a separable, reflexive Banach space.  In particular, $\mathscr{J}\cap \mathfrak{X}$ includes all weakly compact operators between separable spaces. 

\begin{theorem}  For every countable ordinal $\xi$, there exists a separable, reflexive Banach space $S$ such that every member of $\mathscr{J}_{\omega^{\omega^\xi}}\cap \mathfrak{X}$ factors through a subspace of $S$.  In particular, there exists a separable, reflexive Banach space such that  every super weakly compact operator factors through a subspace of $S$.  
\label{theoremC}
\end{theorem}

This result can be compared to the main result of \cite{BF}.  In that paper, a topological space $\mathfrak{L}$ is given such that every operator between separable Banach spaces can be identified with a member of $\mathfrak{L}$.  Thus classes of operators, such as the class of weakly compact operators between separable spaces, can be viewed as subsets of $\mathfrak{L}$, and therefore have some Borel complexity.  In \cite{BF}, it was shown that if $\mathcal{A}\subset \mathfrak{L}$ is an analytic collection consisting of weakly compact operators such that every range space of every operator in $\mathcal{A}$ has a shrinking basis (resp. every range space is $C(2^\nn)$, where $2^\nn$ denotes the Cantor set), then there exists a separable, reflexive space $Z$ through which every member of $\mathcal{A}$ factors.  Our result only allows for factorization through a subspace and not through the whole space. The reason for this difference is that under the assumption that every range space of an operator from $\mathcal{A}$ has a shrinking basis, interpolation allows for each member of $\mathcal{A}$ to be factored through a separable, reflexive Banach space with a basis, and the results of \cite{DF} allow for these spaces to be \emph{complementably} embedded in a universal space.  The complementation of the interpolation spaces allow us to factor through the entire universal space rather than only through a subspace.   

The proof of  Theorem \ref{theoremC} uses several facts from descriptive set theory.  In order to avoid going too far afield, we refer the reader to \cite{D} for the definition of ``coanalytic rank'' and the coding of the class $\textbf{SB}$ of separable Banach spaces and the pertinent properties regarding these topics.  

\begin{proof}[Proof of Theorem \ref{theoremC}] Fix a countable ordinal $\zeta$ such that $\xi\leqslant \omega^{\omega^\zeta}$.  If $A:X\to Y$ is a member of $\mathscr{J}_\xi\cap \mathfrak{X}$, by Theorem \ref{wc} and the remark following Theorem \ref{Heinrich},  $A$  factors through a separable Banach space $Z_A\in \text{Space}(\mathscr{J}_{\omega^{\omega^{\zeta+1}}}\cap \mathfrak{X})=\textbf{SB}\cap \text{Space}(\mathscr{J}_{\omega^{\omega^{\zeta+1}}})$.  We note that $\mathcal{J}$ is a coanalytic rank on the class of separable, reflexive Banach spaces \textbf{REFL} considered as a subset of $\textbf{SB}$. In \cite{Cweakly}, it was shown that $\mathcal{J}$ is a coanalytic rank on the class of all weakly compact operators between separable Banach spaces, and the proof that it is a coanalytic rank on $\textbf{REFL}$ is an inessential modification of this proof.  From this and the properties of coanalytic ranks, $\textbf{SB}\cap \text{Space}(\mathscr{J}_{\omega^{\omega^{\zeta+1}}})$ is Borel in $\textbf{SB}$. By \cite{DF}, there exists a separable, reflexive Banach space $S$ containing isomorphic copies of every member of $\textbf{SB}\cap \text{Space}(\mathscr{J}_{\omega^{\omega^{\zeta+1}}})$. In particular, $S$ contains isomorphic copies of every member of $\{Z_A: A\in \mathscr{J}_\xi\cap \mathfrak{X}\}$, and therefore every member of $\mathscr{J}_\xi\cap \mathfrak{X}$ factors through a subspace of $S$.   

The final sentence follows from the first together with the fact that $\mathscr{J}_\omega$ is the ideal of super weakly compact operators.  \end{proof}

 Before we present the proof of Theorem A we collect a few useful remarks regarding  $\mathcal{J}$.
By the criteria mentioned above, $\mathcal{J}(A)>\omega^{\omega^\xi}$ if and only if there exists $\theta>0$ and a collection $(x_t)_{t\in \ttt_{\omega^{\omega^\xi}}}\subset X$ such that for every $t\in \ttt_{\omega^{\omega^\xi}}$, $(x_{t|_i})_{i=1}^{|t|}\in J(A, \theta)$.

Given a $B$-tree $T$, we let $c(T)$ denote the non-empty subsets of $T$ which are linearly ordered with respect to $\preceq$.  Given $c_1, c_2\in c(T)$, we write $c_1\prec c_2$ if $s\prec t$ for every $s\in c_1$ and $t\in c_2$.  If $S,T$ are $B$-trees, a \emph{block map} is a function $h:S\to c(T)$ such that for every $s_1, s_2\in S\setminus\{\varnothing\}$, $h(s_1)\prec h(s_2)$.  Given two $B$-trees $S,T$, a vector space $X$, and a collection $(x_t)_{t\in T}\subset X$, we say $(y_s)_{s\in S}$ is a \emph{convex block tree} of $(x_t)_{t\in T}$ if there exists a block map $h:S\to T$ such that for every $s\in S$, $y_s\in \text{co}(x_t: t\in h(s))$.  We say that $(x_t)_{t\in T}$ is an $(A, \theta)$-\emph{tree} if for every $t\in T$, $(x_{t|_i})_{i=1}^{|t|}\in J(A, \theta)$.  It is clear that any convex block tree of an $(A, \theta)$-tree is also an $(A, \theta)$-tree. If $A:(\oplus_n X_n)_{\ell_p}\to (\oplus_n Y_n)_{\ell_p}$ is an operator, we let $\mu:(\oplus_n X_n)_{\ell_p}\to \ell_p$ be the map given by $\mu((x_n))=(\|x_n\|)$, and let $\eta:(\oplus_n Y_n)_{\ell_p}\to \ell_p$ be defined similarly.   Let us say that a collection $(x_t)_{t\in T}$ is $\ee$-\emph{close} if for every $s\prec t$, $s,t\in T$, $\|\mu(x_s)-\mu(x_t)\|<\ee$ and $\|\eta(Ax_s)-\eta(Ax_t)\|<\ee$. 

We recall two more facts from \cite{Cweakly}.   For an ordinal $\xi$, we let $\Pi\ttt_\xi=\{(s,t)\in \ttt_\xi\times \ttt_\xi: s\prec t\}$.  

\begin{proposition}Fix an ordinal $\xi$.  \begin{enumerate}[(i)]\item For any finite set $S$ and any function $f:\Pi\ttt_{\omega^{\omega^\xi}}\to S$, there exists a monotone map $\theta:\ttt_{\omega^{\omega^\xi}}\to \ttt_{\omega^{\omega^\xi}}$ such that $f(\theta(\cdot), \theta(\cdot))$ is constant on $\Pi\ttt_{\omega^{\omega^\xi}}$.   \item There exist monotone maps $\phi, \phi':\ttt_{\omega^{\omega^\xi}}\to \ttt_{\omega^{\omega^\xi}}$ such that for every $t\in \ttt_{\omega^{\omega^\xi}}$, $$\phi(t|_1)\prec \phi'(t|_1)\prec \ldots \prec \phi(t)\prec \phi'(t).$$     \end{enumerate}

\label{weaklyprop}
\end{proposition}

We now prove that for any ordinal $\xi$ and $1<p<\infty$, $(\mathscr{J}_{\omega^{\omega^\xi}}, \mathscr{J}_{\omega^{\omega^{\xi+1}}})$ is a $\Sigma_p$-pair, which, in light of Theorem \ref{Heinrich}, will complete Theorem \ref{wc}.  To that end, fix $1<p<\infty$ and a norm $1$ operator  $A:X:=(\oplus X_n)_{\ell_p}\to Y:=(\oplus Y_n)_{\ell_p}$.   Let $\delta_{\ell_p}$ denote the modulus of uniform convexity of $\ell_p$. For $n\in \nn$ and $S\subset \nn$, let $P_n:X\to X_n$, $P_S=\sum_{n\in S}P_n$, $Q_n:Y\to Y_n$, $Q_S=\sum_{n\in S}Q_n$ denote the canonical projections.  

Note that $\mu, \eta$ are norm-preserving, positive homogeneous, and for any vectors $(x_i)_{i=1}^n\subset X$, $$\|\mu(\sum_{i=1}^n x_i)\|\leqslant \|\sum_{i=1}^n \mu(x_i)\|,$$ and the analogous statement holds for vectors in $Y$.  To see the last statement, for each $i$, let $x_i=(x_{ij})_{j=1}^\infty$ and note that $\mu(x_i)=(\|x_{ij}\|)_{j=1}^\infty$.  Then $$\|\mu(\sum_{i=1}^n x_i)\|=\Bigl(\sum_{j=1}^\infty \|\sum_{i=1}^n x_{ij}\|^p\Bigr)^{1/p} \leqslant \Bigl(\sum_{j=1}^\infty \bigl(\sum_{i=1}^n \|x_{ij}\|\bigr)^p\Bigr)^{1/p}= \|\sum_{i=1}^n \mu(x_i)\|.$$

\begin{lemma} Let $(x_t)_{t\in \mathcal{T}_{\omega^{\omega^\xi}}}\subset B_X$ be an $(A, \theta)$-tree. For any $\ee>0$, there exists a convex block tree $(x_t')_{t\in \mathcal{T}_{\omega^{\omega^\xi}}}$ of $(x_t)_{t\in \mathcal{T}_{\omega^{\omega^\xi}}}$ which is $\ee$-close. 
\label{coup de grace}
\end{lemma}

\begin{proof} Seeking a contradiction, suppose that for some $\ee>0$, there is no convex block tree of $(x_t)_{t\in \mathcal{T}_{\omega^{\omega^\xi}}}$ which is $\ee$-close. We define convex block trees $(x_t^i)_{t\in\mathcal{T}_{\omega^{\omega^\xi}}}$ of $(x_t)_{t\in \mathcal{T}_{\omega^{\omega^\xi}}}$ and $k_i, l_i\in \nn\cup \{0\}$ such that for each $i=0, 1, \ldots$,  \begin{enumerate}[(i)]\item $k_i+l_i=i$, \item $(x^i_t)_{t\in \mathcal{T}_{\omega^{\omega^\xi}}}\subset (1-\delta)^{k_i}B_X$, \item $(Ax_t^i)_{t\in \mathcal{T}_{\omega^{\omega^\xi}}}\subset (1-\delta)^{l_i} B_Y$, \end{enumerate} where $\delta=\delta_{\ell_p}(\ee)$.  

We let $x^0_t=x_t$.  Next, suppose that $(x_t^i)_{t\in \mathcal{T}_{\omega^{\omega^\xi}}}$ has been defined and $k_i, l_i$ have been specified. Note that $(x_t^i)_{t\in \mathcal{T}_{\omega^{\omega^\xi}}}\subset (1-\delta)^{k_i}B_X$ and $(Ax^i_t)_{t\in \mathcal{T}_{\omega^{\omega^\xi}}}\subset (1-\delta)^{l_i}B_Y$.  This means that $(\mu(x^i_t))_{t\in \mathcal{T}_{\omega^{\omega^\xi}}}\subset (1-\delta)^{k_i}B_{\ell_p}$ and $(\eta(Ax_t^i))_{t\in \mathcal{T}_{\omega^{\omega^\xi}}}\subset (1-\delta)^{l_i}B_{\ell_p}$.  Define the coloring $f:\Pi \mathcal{T}_{\omega^{\omega^\xi}}\to \{0,1,2\}$ by letting $f(s,t)=0$ if $$\|\mu(x_s^i)-\mu(x_t^i)\|, \|\eta(Ax_s^i)-\eta(Ax_t^i)\|<\ee,$$ $f(s,t)=1$ if $\|\mu(x_s^i)-\mu(x_t^i)\|\geqslant \ee$, and $f(s,t)=2$ if $\|\mu(x_s^i)-\mu(x_t^i)\|<\ee$ and $\|\eta(Ax_s^i)-\eta(Ax_t^i)\|\geqslant \ee$.   Then by Proposition \ref{weaklyprop}, there exists a monotone map $\phi_0:\mathcal{T}_{\omega^{\omega^\xi}}\to \mathcal{T}_{\omega^{\omega^\xi}}$ such that $f(\phi_0(\cdot), \phi_0(\cdot))$ is constant on $\Pi \mathcal{T}_{\omega^{\omega^\xi}}$.   Let $j$ be such that $f(\phi_0(\cdot), \phi_0(\cdot))$ is constantly $j$.  Note that $j\neq 0$, otherwise $(x^i_{\phi_0(t)})_{t\in \mathcal{T}_{\omega^{\omega^\xi}}}$ is an $\ee$-close convex block of $(x_t)_{t\in \mathcal{T}_{\omega^{\omega^\xi}}}$.  For each $t\in \mathcal{T}_{\omega^{\omega^\xi}}$, let $y_t=x^i_{\phi_0(t)}$.  By Proposition \ref{weaklyprop}, we may fix monotone maps $\phi, \phi':\mathcal{T}_{\omega^{\omega^\xi}}\to \mathcal{T}_{\omega^{\omega^\xi}}$ such that for every $t\in \mathcal{T}_{\omega^{\omega^\xi}}$, $\phi(t|_1)\prec \phi'(t|_1)\prec \ldots \prec \phi(t)\prec \phi'(t)$. Note that $h(t)=\{\phi\circ \phi_0(t), \phi'\circ \phi_0(t)\}$ defines a block map.    Let $$x^{i+1}_t= \frac{y_{\phi(t)}+y_{\phi'(t)}}{2}$$ for each $t\in \mathcal{T}_{\omega^{\omega^\xi}}$.  If $j=1$, let $k_{i+1}=1+k_i$ and $l_{i+1}=l_i$. If $j=2$, let $k_{i+1}=k_i$ and $l_{i+1}=1+l_i$.  If $j=1$, then $\|\mu(y_{\phi(t)})-\mu(y_{\phi'(t)})\|\geqslant \ee$ and $\|\mu(y_{\phi(t)})\|, \|\mu(y_{\phi'(t)})\|\leqslant (1-\delta)^{k_i}$, so that $$\frac{\|\mu(y_{\phi(t)})+\mu(y_{\phi'(t)})\|}{2}\leqslant (1-\delta)^{1+k_i}=(1-\delta)^{k_{i+1}}.$$   Since $$\|x^{i+1}_t\|= \|\mu(x^{i+1}_t)\| = \Bigl\|\mu\bigl(\frac{y_{\phi(t)}+y_{\phi'(t)}}{2}\bigr)\Bigr\|\leqslant \frac{\|\mu(y_{\phi(t)})+\mu(y_{\phi'(t)})\|}{2} \leqslant (1-\delta)^{k_{i+1}},$$ we reach the desired conclusion in the case that $j=1$.  If $j=2$, the argument is similar, only we deduce that $$\|Ax^{i+1}t\|=\|\eta(Ax_t^{i+1})\| \leqslant \frac{\|\eta(Ay_{\phi(t)})+\eta(Ay_{\phi(t)})\|}{2}\leqslant (1-\delta)^{l_{i+1}}.$$  This finishes the recursive construction.

Next, fix $i$ such that $(1-\delta)^i<\theta/2$.  Then $k_{2i}+l_{2i}=2i$, and either $k_{2i}\geqslant i$ or $l_{2i}\geqslant i$.  If $k_{2i}\geqslant i$, note that since $(x^{2i}_t)_{t\in \mathcal{T}_{\omega^{\omega^\xi}}}\subset (1-\delta)^{k_{2i}}B_X\subset (1-\delta)^i B_X$ and since $\|A\|=1$, $(Ax^{2i}_t)_{t\in \mathcal{T}_{\omega^{\omega^\xi}}}\subset (1-\delta)^i B_Y$.   If $l_{2i}\geqslant i$, $(Ax^{2i}_t)_{t\in \mathcal{T}_{\omega^{\omega^\xi}}}\subset (1-\delta)^{l_{2i}}B_Y\subset (1-\delta)^iB_Y$.  Then for any $s,t\in \mathcal{T}_{\omega^{\omega^\xi}}$ with $s\prec t$, $$\|Ax_s^{2i}-Ax_t^{2i}\|\leqslant 2(1-\delta)^i<\theta. $$ But since $(x^{2i}_u)_{u\in \mathcal{T}_{\omega^{\omega^\xi}}}$ is a convex block tree of an $(A, \theta)$-tree, it must be an $(A, \theta)$-tree as well, and we reach a contradiction.  
\end{proof}

\begin{corollary} Suppose that for some ordinal $\xi$, $\mathcal{J}(A)>\omega^{\omega^{\xi+1}}$.  Then there exist $k,l\in \nn$ such that $\mathcal{J}(Q_lAP_k)>\omega^{\omega^\xi}$.  

\end{corollary}

\begin{proof} We may assume without loss of generality that $\|A\|=1$.   Fix $\theta>0$ such that $\mathcal{J}(A, \theta)>\omega^{\omega^{\xi+1}}$.  We may fix a collection $(x_t)_{t\in \mathcal{T}_{\omega^{\omega^{\xi+1}}}}\subset B_X$ which is $(A, \theta)$-separated. By Lemma \ref{coup de grace}, we may assume that $(x_t)_{t\in \mathcal{T}_{\omega^{\omega^{\xi+1}}}}$ is $\ee$-close, where $\ee=\theta/9$.  Fix any $t\in \mathcal{T}_{\omega^{\omega^{\xi+1}}}$ such that $o(\mathcal{T}_{\omega^{\omega^{\xi+1}}}(t))>\omega^{\omega^\xi}$, as we may, since $o(\mathcal{T}_{\omega^{\omega^{\xi+1}}})=\omega^{\omega^{\xi+1}}$.  Fix any monotone $\psi:\mathcal{T}_{\omega^{\omega^\xi}}\to \{s\in \mathcal{T}_{\omega^{\omega^{\xi+1}}}: t\prec s\}$. Such a map exists, since we may first fix a monotone map $\psi':\mathcal{T}_{\omega^{\omega^\xi}}\to \mathcal{T}_{\omega^{\omega^{\xi+1}}}(t)$ simply by comparing orders of these trees, and let $\psi(s)=t\cat \psi'(s)$.    For each $s\in \mathcal{T}_{\omega^{\omega^\xi}}$, let $z_s=x_{\psi(s)}$.  

Fix some $n\in\nn$ such that $\|P_{(n, \infty)}x_t\|<\ee$ and $\|Q_{(n,\infty)}Ax_t\|<\ee$.  Let $\pi:\ell_p\to \ell_p$ denote the tail projection $\pi\sum_{i=1}^\infty a_ie_i=\sum_{i=n+1}^\infty a_ie_i$ and note that $\mu\circ P_{(n, \infty)}=\pi\circ \mu$ and $\eta\circ Q_{(n,\infty)}=\pi\circ \eta$.  Then for any $s\in \mathcal{T}_{\omega^{\omega^\xi}}$, \begin{align*} \|P_{(n, \infty)}z_s\| & =\|\pi \mu(z_s)\|=\|\pi\mu(x_{\psi(s)})\|\leqslant \|\pi\mu(x_t)\|+\|\mu(x_t)-\mu(x_{\psi(s)})\| \\ & <\|P_{(n, \infty)}x_t\|+\ee<2\ee.\end{align*}  Similarly, $$ \|Q_{(n,\infty)}Az_s  \| \leqslant \|Q_{(n,\infty)}Ax_t\|+\ee<2\ee.$$  From this we deduce that $$\|Az_s-Q_{[1,n]}AP_{[1,n]}z_s\| \leqslant \|Az_s-Q_{[1,n]}Az_s\|+\|Q_{[1,n]}A\|\|z_s-P_{[1,n]}z_s\|<4\ee.$$  From this it follows that $(z_s)_{s\in \mathcal{T}_{\omega^{\omega^\xi}}}$ is a $(Q_{[1,n]}AP_{[1,n]}, \theta/9)$-tree, which yields $\mathcal{J}(Q_{[1,n]}AP_{[1,n]})>\omega^{\omega^\xi}$.  In order to see that this is a $(Q_{[1,n]}AP_{[1,n]}, \theta/9)$-tree, we first note that it is an $(A, \theta)$-tree.  Fix $s_0,s_1\in \mathcal{T}_{\omega^{\omega^\xi}}$ with $s_0\prec s_1$ and fix $x=\sum_{s\preceq s_0}a_sz_s\in \text{co}(z_s: s\preceq s_0)$, and $y=\sum_{s_0\prec s\preceq s_1} a_sz_s\in \text{co}(z_s: s_0\prec s\preceq s_1)$. Then \begin{align*} \|Q_{[1,n]}AP_{[1,n]}x-Q_{[1,n]}AP_{[1,n]}y\| & \geqslant \|Ax-Ay\| - \sum_{s\preceq s_0}a_s \|Az_s-Q_{[1,n]}AP_{[1,n]}z_s\| \\ & - \sum_{s_0\prec s\preceq s_1} a_s\|Az_s-Q_{[1,n]}AP_{[1,n]}z_s\| \\ & \geqslant \theta-4\ee-4\ee= \theta-8\theta/9=\theta/9.  \end{align*}

Since $Q_{[1,n]}AP_{[1,n]}=\sum_{k,l\leqslant n}Q_l AP_k$ and $\mathscr{J}_{\omega^{\omega^{\xi+1}}}$ is closed under finite sums, we deduce the result.

\end{proof}

\section{$\ell_1^\xi$ Spreading models and $\xi$-Banach-Saks operators}

For an ordinal $0<\xi<\omega_1$, a bounded sequence $(x_n)$ in the Banach space $X$ is said to be an $\ell_1^\xi$ spreading model if there exists $K>0$ such that for every $E\in \sss_\xi$ and every set of scalars $(a_i)_{i\in E}$, $$K^{-1}\sum_{i\in E}|a_i|\leqslant \|\sum_{i\in E}a_ix_i\|.$$ 
For every $0<\xi<\omega_1$, we let $\mathfrak{SM}_1^\xi$ denote those operators $A:X\to Y$ such that for any $\ell_1^\xi$-spreading model $(x_i)\subset X$, $(Ax_i)$ is not an $\ell_1^\xi$-spreading model.  We deduce that $\mathfrak{SM}_1^\xi$ is injective and surjective in a way similar to that of the weakly compact operators.

The main result of this section is the Theorem A from the introduction. Before providing the proof we introduce new classes of operators called we call the $\xi$-Banach Saks operators. These classes naturally generalize the well-known class of Banach-Saks 
operators and coincide with to two other classes of operators studied in \cite{BF,BCFW}.

In \cite{BCFW} it is shown that for each $0<\xi<\omega_1$, the class $\mathfrak{WC}^\xi:=\mathfrak{SM}_1^\xi\cap \mathscr{J}$ coincides with the classes of $\mathcal{S}_\xi$-weakly compact operators from \cite{BF}. We now define the class of $\xi$-Banach-Saks operators for $0<\xi< \omega_1$. 

Fix an operator $A:X\to Y$ and suppose that $(x_n)$ is a bounded sequence in $X$.  We may define $BS((x_n),A)$ to be the smallest countable ordinal $\xi$ (if any such exists) such that $(Ax_n)$ is $\xi$-convergent to a member of $Y$. If no such countable ordinal exists, we write $BS((x_n), A)=\omega_1$.  It is shown in \cite{AMT} that such a countable ordinal exists provided $(Ax_n)$ has a weakly convergent subsequence. Conversely, if $(Ax_n)$ is $\xi$-convergent, it has a subsequence with convex blocks (the coefficients of which are given by $(\xi^M_n)$ for some $M$) converging in norm to some vector $y\in Y$, and therefore this subsequence converges weakly to $y$.  Thus there exists some countable $\xi$ such that $(Ax_n)$ is $\xi$-convergent if and only if $(Ax_n)$ has a weakly convergent subsequence. This motivates the following definition which was not isolated in \cite{BCFW} but was implicitly contained.

\begin{definition} For $\xi<\omega_1$, we  say $A:X\to Y$ is $\xi$-\emph{Banach-Saks} provided that for every bounded sequence $(x_n)$ in $X$, $BS((x_n), A)\leqslant \xi$. Let  $\mathfrak{BS}_\xi$ denote the class of $\xi$-Banach-Saks operators. \end{definition}

For completeness we recall the definition of $\mathcal{S}_\xi$-weakly compact.

\begin{definition}
For $\xi<\omega_1$, we say an operator $A:X\to Y$ is $\mathcal{S}_\xi$-weakly compact if it fails to have the following property: There exists a constant $K>0$ and a seminormalized basic sequence $(x_n)\subset X$ such that for every $E\in \mathcal{S}_\xi$ and all scalars $(a_n)_{n\in E}$, $\|\sum_{n\in E} a_n Ax_n\|\geqslant K \|\sum_{n\in E} a_n s_n\|$. Here $(s_n)$ is the summing basis, the norm of which is given by $$\|\sum_{n=1}^k a_ns_n\|=\max_{1\leqslant l\leqslant k} |\sum_{n=1}^l a_n|.$$ 
 \end{definition}

In summary, we have following theorem whose proof can be found in \cite{BCFW}

\begin{theorem}
Let $\xi$ be an ordinal with $0<\xi<\omega_1$. The following classes of operators are the same.
\begin{enumerate}[(i)]
\item The $\xi$-Banach-Saks operators $\mathfrak{BS}_\xi$.
\item The class $\mathfrak{WC}^\xi=\mathfrak{SM}_1^\xi\cap \mathscr{J}$.
\item The $\mathcal{S}_\xi$-weakly compact operators.
\end{enumerate}
\end{theorem}

 Our previous discussion guarantees that if $A$ is $\xi$-Banach-Saks, it is weakly compact.    A standard ``overspill'' argument guarantees that if $X$ is separable, then the converse is also true.  That is, if $A:X\to Y$ is weakly compact and $X$ is separable, then there exists $\xi<\omega_1$ such that $A$ is $\xi$-Banach-Saks.  However, there are examples of operators on non-separable domains which are weakly compact but not $\xi$-Banach-Saks for any $\xi<\omega_1$. 

We summarize this discussion in the following. Items $(iii)$ and $(iv)$ follow from our description of the level $(1^M_n)$ of the repeated averages hierarchy. 

\begin{proposition} Fix an operator $A:X\to Y$.  \begin{enumerate}[(i)]\item If $A$ is $\xi$-Banach-Saks for some $\xi<\omega_1$, $A$ is weakly compact. \item If $A$ has separable range and is weakly compact, there exists $\xi<\omega_1$ such that $A$ is $\xi$-Banach-Saks.   \item $\mathfrak{BS}_1$ coincides with the class of Banach-Saks operators.   \end{enumerate}
\label{paul walker}
\end{proposition}

Below we restate Theorem A.

\begin{theorem} Let $\xi$ with $0<\xi<\omega_1$. Then $\mathfrak{SM}_1^\xi$ and $\mathfrak{BS}_\xi$ have the factorization property.  
\end{theorem}

The fact that $\mathfrak{BS}_1$ has the factorization property is due to Beauzamy \cite{Beauzamy2}. We make one final remark before presenting the proof of Theorem A. 

\begin{rem}\upshape
Fix $n\in \nn$ and let $M=(in)_{i=1}^\infty$. Then an easy proof by induction yields that $\sss_\xi[\aaa_n](M)\subset \sss_\xi$ and if $(x_i)$ is an $\ell_1^\xi$ spreading model with constant $K$, and if $E_1<E_2<\ldots$ are subsets of $M$ with $|E_i|= n$, then the blocking $(n^{-1}\sum_{i\in E_j} x_i)_j$ of $(x_i)$ is also an $\ell_1^\xi$ spreading model with constant $K$. Here 
$$\sss_\xi[\aaa_n]=\Bigl\{\overset{k}{\underset{i=1}{\bigcup}}E_i: k\in \nn, E_i\in \aaa_n, E_1<\ldots <E_k, (\min E_i)_{i=1}^k\in \sss_\xi\Bigr\}.$$
and
$$\sss_\xi[\aaa_n](M)=\{(m_i)_{i\in E}: E\in \sss_\xi[\aaa_n]\}.$$
\end{rem}

\begin{proof}[Proof of Theorem A]

To prove the theorem we will again apply Theorem \ref{Heinrich}. That is, we must show that $(\mathfrak{SM}_\xi, \mathfrak{SM}_\xi)$ is a $\Sigma_p$-pair for any $1<p<\infty$ and $0<\xi<\omega_1$.  Combining this with the fact that $(\mathscr{J}, \mathscr{J})$ is a $\Sigma_p$-pair for any $1<p<\infty$ (which is a consequence of Theorem \ref{wc}) yields that for any $1<p<\infty $ and $0<\xi<\omega_1$, $(\mathfrak{BS}_\xi, \mathfrak{BS}_\xi)$ is a $\Sigma_p$-pair, since $\mathfrak{BS}_\xi=\mathscr{J}\cap \mathfrak{SM}_1^\xi$.   

 Suppose that $A:(\oplus X_n)_{\ell_p}\to (\oplus Y_n)_{\ell_p}$ is an operator which preserves an $\ell_1^\xi$ spreading model. The fact that $(\mathfrak{SM}_\xi, \mathfrak{SM}_\xi)$ is a $\Sigma_p$-pair is implied by the following three items: 
 \begin{enumerate}[(i)]
 \item There exists $m\in \nn$ such that $AP_{[1,m]}$ preserves an $\ell_1^\xi$ spreading model. 
 \item There exists $n\in \nn$ such that $Q_{[1,n]}A$ preserves an $\ell_1^\xi$ spreading model. 
 \item There exist $i,j\in \nn$ such that $Q_j AP_i$ preserves an $\ell_1^\xi$ spreading model.  
 \end{enumerate}

 Assume without loss of generality that $\|A\|= 1$. In the proof, let $X=(\oplus X_n)_{\ell_p}$ and $Y=(\oplus Y_n)_{\ell_p}$. As in the previous section, let $\mu:X\to \ell_p$ denote the function $\mu((x_n))=(\|x_n\|)$ and $\eta:Y\to \ell_p$ denote the function $\eta((y_n))=(\|y_n\|)$.    Assume $0<\xi<\omega_1$, $(x_i)\subset B_X$, and $\ee>0$ are such that for every $E\in \sss_\xi$ and all scalars $(a_i)_{i\in E}$, $\|\sum_{i\in E} a_iAx_i\|\geqslant 4\ee \sum_{i\in E}|a_i|$.  For $m\in \nn$, let $p_{(m,\infty)}:\ell_p\to \ell_p$ denote the tail projection in $\ell_p$.

Write $x_i=(x_{ij})_{j=1}^\infty$ with $x_{ij}\in X_j$ and $Ax_i=(y_{ij})_{j=1}^\infty$, $y_{ij}\in Y_j$.   By passing to a subsequence, we may assume that $\mu(x_i)\underset{w}{\to}\mu_0$ and $\eta(y_i)\underset{w}{\to} \eta_0$.   Fix $m,n\in \nn$ such that $\|p_{(m, \infty)} \mu_0\|<\ee$ and $\|p_{(n,\infty)}\eta_0\|<\ee$.    By passing to a subsequence once more, we may assume there exist block sequences $(u_i), (v_i)$ in $B_{\ell_p}$ such that \begin{enumerate}[(i)]\item for every $i\in \nn$, $\|\mu(x_i)-(\mu_0+u_i)\|<\ee$, \item for every $i\in \nn$, $\|\eta(y_i)-(\eta_0+v_i)\|<\ee$, \item $\min\supp(u_1)>m$,  \item $\min \supp(v_1)>n$.   \end{enumerate} Fix a natural number $k$ such that $1/k^{1/q}<\ee$, where $1/p+1/q=1$.    Fix $M=(ik)_{i=1}^\infty$ and recall that $\sss_\xi[\aaa_k](M)\subset \sss_\xi$ and $B_1<B_2<\ldots$ such that $|B_i|=k$ and $B_i\subset M$.  Let $g_i=\frac{1}{k}\sum_{j\in B_i} x_j$. Of course, $g_i\in B_X$.  Note that since $\sss_\xi[\aaa_k](M)\subset \sss_\xi$, for any $E\in \sss_\xi$ and any scalars $(a_i)_{i\in E}$, $$\|\sum_{i\in E} a_iAg_i\|\geqslant 4\ee \sum_{i\in E}|a_i|.$$  

\begin{claim} For every $i\in \nn$, $\|g_i-P_{[1, m]}g_i\|<3\ee$ and $\|Ag_i-Q_{[1,n]}Ag_i\|<3\ee$.   
\end{claim}

In order to see the claim, recall that $x_i=(x_{ij})_{j=1}^\infty$.  Then  \begin{align*} \|g_i-P_{[1,m]}g_i\| & = \frac{1}{k}\Bigl(\sum_{j=m+1}^\infty \bigl\|\sum_{l\in B_i} x_{lj}\bigr\|^p\Bigr)^{1/p} \leqslant \frac{1}{k}\Bigl(\sum_{j=m+1}^\infty \bigl(\sum_{l\in B_i} \|x_{lj}\|\bigr)^p\Bigr)^{1/p} \\ & = \frac{1}{k}\bigl\|\sum_{l\in B_i} p_{(m, \infty)} \mu(x_l)\bigr\| \\ & \leqslant \frac{1}{k}\sum_{l\in B_i}\|p_{(m, \infty)}(\mu(x_l)-(\mu_0+u_l))\| + \frac{1}{k}\bigl\|\sum_{l\in B_i} p_{(m, \infty)} (\mu_0+u_l)\bigr\| \\ & \leqslant \frac{1}{k}\cdot \ee |B_i| + \|p_{(m, \infty)} \mu_0\| + \frac{1}{k}\bigl\|\sum_{l\in B_i} u_l\bigr\| \\ & < 2\ee + \frac{1}{k}\bigl(\sum_{l\in B_i} \|u_l\|^p\bigr)^{1/p} <3\ee.\end{align*}

The proof that $\|Ag_i-Q_{[1,n]}Ag_i\|<3\ee$ is similar, and we deduce the claim.   

Then for any $E\in \sss_\xi$ and any scalars $(a_i)_{i\in E}$, \begin{align*} \|\sum_{i\in E} a_iA P_{[1,m]} g_i\| & \geqslant \|\sum_{i\in E} a_i Ag_i\|- \sum_{i\in E}|a_i|\|g_i-P_{[1,m]}g_i\| \\ & \geqslant 4\ee\sum_{i\in E}|a_i|- 3\ee\sum_{i\in E}|a_i|=\ee\sum_{i\in E}|a_i|.\end{align*} Similarly, \begin{align*} \|\sum_{i\in E} a_iQ_{[1,n]}A g_i\| & \geqslant \|\sum_{i\in E} a_i  Ag_i\|- \sum_{i\in E}|a_i|\|Ag_i-Q_{[1,n]}Ag_i\| \\ & \geqslant 4\ee\sum_{i\in E}|a_i|- 3\ee\sum_{i\in E}|a_i|=\ee\sum_{i\in E}|a_i|.\end{align*} 

This means that $(g_i)$, $(AP_{[1,m]}g_i)$, and $(Q_{[1,n]}Ag_i)$ are all $\ell_1^\xi$ spreading models, yielding $(i)$ and $(ii)$.   

For $(iii)$, first suppose that $A$ preserves an $\ell_1^\xi$ spreading model.  Then by $(i)$, there exists $m\in \nn$ such that $AP_{[1,m]}$ preserves an $\ell_1^\xi$ spreading model.  By $(ii)$ applied to $AP_{[1,m]}$, there exists $n\in \nn$ such that $Q_{[1,n]}AP_{[1,m]}$ preserves an $\ell_1^\xi$ spreading model.   But $$Q_{[1,n]}AP_{[1,m]}=\sum_{i=1}^n \sum_{j=1}^m Q_iAP_j.$$ Since this is a finite sum, we know that if for each $1\leqslant i\leqslant n$ and $1\leqslant j\leqslant m$, if $Q_iAP_j$ fails to preserve an $\ell_1^\xi$ spreading model, then $Q_{[1,n]}AP_{[1,m]}$ fails to preserve an $\ell_1^\xi$ spreading model.  Thus if $Q_{[1,n]}AP_{[1,m]}$ preserves an $\ell_1^\xi$ spreading model, there exists $(i,j)\in \{1, \ldots, n\}\times \{1, \ldots m\}$ such that $Q_iAP_j$ preserves an $\ell_1^\xi$ spreading model.  
\end{proof}

\section{Rosenthal operators}

\subsection{Factorization of Rosenthal operators}

Given an operator $A:X\to Y$ and $K>0$, we let $T_1(A,K)$ denote the set consisting of the empty sequence and those sequences $(x_i)_{i=1}^n\subset B_X$ such that for all scalars $(a_i)_{i=1}^n$, $K\|\sum_{i=1}^n a_iAx_i\|\geqslant \sum_{i=1}^n |a_i|$.  We then let $\textbf{NP}_1(A)=\sup_{K>0} o(T_1(A,K))$, where $o(T)$ denotes the order of a tree. The operator $A$ is a Rosenthal operator if and only if $\textbf{NP}_1(A)$ is an ordinal.   Given an ordinal $\xi$, we let $\mathfrak{NP}^\xi_1$ denote the class of operators $A$ such that $\textbf{NP}_1(A)\leqslant \omega^\xi$.  It was shown in \cite{BCFW} that for every ordinal $\xi$, $\mathfrak{NP}^{\omega^\xi}_1$ is a closed operator ideal.  The class corresponding to $\xi=0$ is class of super Rosenthal operators.  Moreover, there exist Rosenthal operators with arbitrarily large $\textbf{NP}_1$ index, and there exist Rosenthal operators on separable domains with arbitrarily large, countable $\textbf{NP}_1$ index.  Injectivity and surjectivity of these classes are easily established. The main result of this section is the following.

This result is a restatement of Theorem B item (i) from the introduction.

\begin{theorem}\begin{enumerate}[(i)]\item For every ordinal $\xi$, $\mathfrak{NP}_1^{\omega^\xi}$ has the $\mathfrak{NP}^{\omega^\xi+1}$ factorization property. \item $\mathfrak{NP}_1^{\omega^\xi}$ has the factorization property if and only if $\xi$ has uncountable cofinality.  \end{enumerate} 

\label{Rosenthal theorem}
\end{theorem}

As is now routine, we only need to show the following in order to deduce Theorem \ref{Rosenthal theorem}$(i)$.  

\begin{proposition} Suppose that for every $m,n\in \nn$, if $\textbf{\emph{NP}}_1(Q_n AP_m)\leqslant \omega^{\omega^\xi}$, then $\textbf{\emph{NP}}_1(A)\leqslant \omega^{\omega^\xi+1}$.  

\end{proposition}

\begin{proof} Again, assume $\|A\|=1$.  We will show something stronger than what is stated under slightly different assumptions.  Assume that for any $m,n\in \nn$, $\textbf{NP}_1(Q_{[1,n]}AP_{[1,m]})\leqslant \omega^\xi$.  We will show that $\textbf{NP}_1(A)\leqslant \omega^{\xi+1}$.  This will imply the proposition as stated.  Indeed, since $\mathfrak{NP}_1^{\omega^\xi}$ is closed under finite sums, it follows that if $\textbf{NP}_1(Q_nAP_m)\leqslant \omega^{\omega^\xi}$ for every $m,n\in \nn$, then $\textbf{NP}_1(Q_{[1,n]}AP_{[1,m]})\leqslant \omega^{\omega^\xi}$, for every $m,n\in\nn$.    

To obtain a contradiction, suppose that $\textbf{NP}_1(A)>\omega^{\xi+1}$ and  $\textbf{NP}_1(Q_{[1,n]}AP_{[1,m]})\leqslant \omega^\xi$ for every $m,n\in \nn$.  Fix $K\geqslant 1$ such that $o(T_1(A,K))>\omega^{\xi+1}$.  We fix $n\in \nn$ and $1=r_0<\ldots <r_n$, $1=s_0<s_1<\ldots <s_n$, and a member $(y_i)_{i=1}^n$ of $T_1(A,K)$ such that for each $1\leqslant i\leqslant n$, \begin{enumerate}[(i)]\item $1/n^{1/q}<1/5K$, \item $\|Q_{[s_0, s_{i-1}]}AP_{[r_0, r_{i-1}]}y_i\|<1/5K $, \item $\|P_{(r_i, \infty)}y_i\|<1/5K$, \item $\|Q_{(s_i, \infty)}AP_{[r_0, r_{i-1}]}y_i\|<1/5K$.  \end{enumerate}

We first finish the proof, and then show how to choose the $y_i$ vectors.  Let $u_i=P_{[r_0, r_{i-1}]}y_i$, $v_i=P_{(r_{i-1}, r_i]} y_i$, and $w_i=P_{(r_i, \infty)}y_i$.  Furthermore, let $u_i'=Q_{[s_0, s_{i-1}]} Au_i$, $u_i''=Q_{(s_{i-1}, s_i]}Au_i$, and $u_i'''=Q_{(s_i, \infty)}Au_i$.   Then $Ay_i=Av_i+Aw_i+u_i'+u_i''+u_i'''$.  Note that the vectors $v_1, \ldots, v_n$ are successively supported in $(\oplus X_i)_{\ell_p}$ and have norm at most $1$, since each $y_i$ has norm at most $1$, so that $\frac{1}{n}\|\sum_{i=1}^n v_i\|\leqslant n^{1/p}/n=1/n^{1/q}<1/5K$.  Since $\|A\|=1$, $\|\frac{1}{n}\sum_{i=1}^n Av_i\|<1/5K$.  

Similarly, the vectors $u_1'', \ldots, u_n''$ are successively supported, so that $\|\frac{1}{n}\sum_{i=1}^n u_i''\|<1/5K$.  

Furthermore, by our choices, $\|u_i'\|, \|w_i\|, \|u_i'''\|<1/5K$, so that $\frac{1}{n}\|\sum_{i=1}^n u_i'\|$, $\frac{1}{n}\|\sum_{i=1}^n Aw_i\|$, $\frac{1}{n}\|\sum_{i=1}^n u_i'''\|<1/5K$.    From this it follows that \begin{align*} \frac{1}{n}\|\sum_{i=1}^n Ay_i\| & \leqslant \frac{1}{n}\|\sum_{i=1}^n Av_i\|+\frac{1}{n}\|\sum_{i=1}^n Aw_i\| \\ & + \frac{1}{n}\|\sum_{i=1}^n u_i'\| + \frac{1}{n}\|\sum_{i=1}^n u_i''\|+ \frac{1}{n}\|\sum_{i=1}^n u_i'''\|\\ & <1/K.\end{align*}  But this is a contradiction, since $(y_i)_{i=1}^n\in T_1(A,K)$, $\|\sum_{i=1}^n \frac{1}{n}Ay_i\|\geqslant 1/K$, and this contradiction yields the desired conclusion.  

We return to the choice of the vectors $y_i$. Let us recall some notation and facts mentioned above. Given a tree $T$ and an ordinal $\zeta$, $T^\zeta$ will denote the $\zeta^{th}$ derived tree.  Given a sequence $t\in T$, we let $T(t)$ denote those non-empty sequences $s$ such that the concatenation $t\cat s\in T$, which is a $B$-tree.  We note that if $t\in T$,  $t\in T^\zeta$ if and only if $o(T(t))\geqslant \zeta$.  Moreover, $T^\zeta(t)=(T(t))^\zeta$ for any $t\in T$ and any ordinal $\zeta$.  We also note that for any $r,s\in \nn$, if $T$ is a $B$-tree in $B_X$ with $o(T)\geqslant \omega^\xi$, then there exist $(x_i)_{i=1}^j\in T$ and scalars $(a_i)_{i=1}^j$ such that $\sum_{i=1}^j |a_i|=1$ and $\|\sum_{i=1}^j a_iQ_{[1,s]}AP_{[1,r]}x_i\|<1/5K$.  Indeed, if it were not so, then $o(T_1(Q_{[1,s]}AP_{[1,r]}, 5K))\geqslant o(T\cup \{\varnothing\})>\omega^\xi$, contradicting the hypothesis that $\textbf{NP}_1(Q_{[1,s]}AP_{[1,r]})\leqslant \omega^\xi$.  

 First fix $n\in \nn$ such that $1/n^{1/q}<1/5K$.  Let $r_0=s_0=1$.  Fix any $y_1$ such that the length one sequence $(y_1)$ is a member of $T_1(A,K)^{\omega^\xi (n-1)}$.  We may do this, since $\omega^\xi(n-1)<\omega^{\xi+1}$.  Next, assume that $(y_1, \ldots, y_k)\in T_1(A,K)^{\omega^\xi(n-k)}$, $r_0<\ldots <r_k$, and $s_0<\ldots <s_k$ have been chosen for some $k<n$.  Let $t=(y_1, \ldots, y_k)$ and let $T=T_1(A,K)^{\omega^\xi(n-k-1)}(t)$.  Note that $o(T)\geqslant \omega^\xi$ by our remarks above.  Then there exist $(x_i)_{i=1}^j\in T$ and scalars $(a_i)_{i=1}^j$ such that $\sum_{i=1}^j |a_i|=1$ and $\|\sum_{i=1}^j a_iQ_{[1, s_k]}AP_{[1,r_k]} x_i\|<1/5K$.  Let $y_{k+1}=\sum_{i=1}^j a_ix_i$.  Choose $r_{k+1}>r_k$ such that $\|P_{(r_{k+1}, \infty)}y_{k+1}\|<1/5K$ and $s_{k+1}>s_k$ such that $\|Q_{(s_{k+1}, \infty)}AP_{[1, r_k]} y_{k+1}\|<1/5K$.  This completes the recursive construction, since $(y_i)_{i=1}^{k+1}\in T_1(A,K)^{\omega^\xi(n-k-1)}$.

\end{proof}

\begin{remark} For any $1<p<\infty$ and an operator $A:X\to Y$, we may define the index $\textbf{NP}_p(A)$ to be the supremum over all $K>0$ of the orders of the trees $T_p(A,K)$ consisting of the empty sequences together with those sequences $(x_i)_{i=1}^n$ such that for every $(a_i)_{i=1}^n\in S_{\ell_p^n}$, $$\|\sum_{i=1}^n a_ix_i\|\leqslant 1, \hspace{5mm} K\|\sum_{i=1}^n a_iAx_i\|\geqslant 1.$$  We define $\mathfrak{NP}_p^\xi$ to be the class of those operators $A$ such that $\textbf{NP}_p(A)\leqslant \omega^\xi$.

Arguing as in the previous proof, we may deduce that for any $1\leqslant p<q<\infty$, if $A:(\oplus_n X_n)_{\ell_q}\to (\oplus Y_n)_{\ell_q}$ is such that for all $m,n\in \nn$, $\textbf{NP}_p(Q_{[1,n]}AP_{[1,m]})\leqslant \omega^\xi$, then $\textbf{NP}_p(A)\leqslant \omega^{\xi+1}$.  However, this does not yield a factorization result, since the class $\mathfrak{NP}_p^\xi$ is not surjective. Indeed, the first step of the proof of Theorem \ref{Heinrich} is to pass from an operator $A:X\to Y$ to the induced operator $B:X/\ker(A)\to Y$, and the estimate we used was actually on the $\textbf{NP}_1$ index of $B$.   For $1<p<\infty$, if $A:\ell_1\to \ell_p$ is a quotient map, $\textbf{NP}_p(A)\leqslant \textbf{NP}_p(\ell_1)\leqslant \omega^2$ \cite{Cproximity}, while the induced operator is the identity on $\ell_p$ and therefore has $\textbf{NP}_p$ index $\infty$.

\end{remark}

\begin{proof}[Proof of Theorem \ref{Rosenthal theorem}$(ii)$] If $\textbf{NP}_1(A)\leqslant \omega^{\omega^\xi}$ and $\xi$ has uncountable cofinality, then the inequality must be strict \cite{BCFW}, and there exists $\zeta<\xi$ such that $\textbf{NP}_1(A)\leqslant \omega^{\omega^\zeta}$.  Then $A$ factors through a Banach space $Z$ such that $\textbf{NP}_1(Z)\leqslant \omega^{\omega^\zeta+1}<\omega^{\omega^\xi}$. 

Next, suppose that $\xi$ has countable cofinality. If $\xi$ is a limit ordinal, it was shown in \cite{BCFW} that there exists an operator $A$ with $\textbf{NP}_1(A)=\omega^{\omega^\xi}$.  It was shown in \cite{Cproximity} that there is no Banach space with this $\textbf{NP}_1$ index.  Thus any $Z$ through which $A$ factors must satsify $\textbf{NP}_1(Z)>\omega^{\omega^\xi}$.  We must consider the cases that $\xi=0$ and $\xi$ is a successor. First assume that $\xi=\zeta+1$. It was shown in \cite{BCFW} that for every $n$, there exists a Banach space $X_n$ with $o(T_1(X_n, 1))>\omega^{\omega^\zeta 2^n}$ and $\textbf{NP}_1(X_n)=\omega^{\omega^\zeta 2^n +1}$.   Moreover, it was shown there that the operator $A:(\oplus X_n)_{\ell_2}\to (\oplus X_n)_{\ell_2}$ such that $A_n:=A|_{X_n}=2^{-n} I_{X_n}$ satisfies $\textbf{NP}_1(A)=\omega^{\omega^\xi}$.  It follows from the construction that $o(T_1(A, 2^n))\geqslant o(T_1(X_n, 1))>\omega^{\omega^\zeta 2^n}$.  We will show that this $A$ does not factor through any Banach space $Z$ with $\textbf{NP}_1(Z)=\omega^{\omega^\xi}$.  To that end, note that if $A$ factors through $Z$, there exists a constant $C$ such that $o(T_1(A,K))\leqslant o(T_1(Z, CK))$ for all $K>0$.  Suppose that $\textbf{NP}_1(Z)\leqslant \omega^{\omega^\xi}$, which implies that $o(T_1(Z, 2))<\omega^{\omega^\xi}$.  Since $\sup_m \omega^{\omega^\zeta m}=\omega^{\omega^\xi}$, there exists $m\in \nn$ such that $o(T_1(Z, 2))<\omega^{\omega^\zeta m}$.  It was shown in \cite{Cproximity} that for any $n\in \nn$, $$o(T_1(Z, 2^n))\leqslant o(T_1(Z,2))^n<(\omega^{\omega^\zeta m})^n=\omega^{\omega^\zeta mn}.$$  There exist $n, n_0\in \nn$ such that $2^{n_0}>C$ and $2^n>m(n+n_0)$.  Then \begin{align*} o(T_1(A,2^n)) & > \omega^{\omega^\zeta 2^n}> \omega^{\omega^\zeta m(n+n_0)} \\ & \geqslant o(T_1(Z, 2^{n+n_0})) \geqslant o(T_1(Z, C2^n)) \geqslant o(T_1(A, 2^n)),\end{align*} a contradiction.

For the $\xi=0$ case, we may appeal to our compact diagonal operator $A$ on $c_0$ having no non-trivial Rademacher cotype.  Since this operator is compact and not finite rank, $\textbf{NP}_1(A)=\omega$. However, $c_0$, and therefore $\ell_1$, is finitely representble in any Banach space through which $A$ factors, whence the $\textbf{NP}_1$ index of any space through which $A$ factors exceeds $\omega$.  
\end{proof}

We now prove restate and prove the second part of Theorem C from the introduction. 

\begin{theorem} For every countable ordinal $\xi$, there exists a separable Banach space $S$ containing no copy of $\ell_1$ such that every member of $\mathfrak{NP}_1^{\omega^\xi}\cap \mathfrak{X}$ factors through a quotient of $S$.   
\end{theorem}

\begin{proof} By Theorem \ref{Rosenthal theorem}, every operator $A:X\to Y$ lying in $\mathfrak{NP}_1\cap \mathfrak{X}$ factors through a separable Banach space $Z_A$ with $\textbf{NP}_1(Z_A)\leqslant \omega^{\omega^\xi+1}=:\gamma$.  By a result of Dodos \cite{Dodos}, there exists a separable Banach space $S$ containing no copy of $\ell_1$ such that every Banach space $Z$ with $\textbf{NP}_1(Z)\leqslant \gamma$ is isomorphic to a quotient of $S$.   In particular, every member of $\{Z_A: A\in \mathfrak{NP}_1\cap \mathfrak{X}\}$ is isomorphic to a quotient of $S$.

\end{proof}

\begin{remark} We note that a universality result analogous to  Theorem C is not possible for $\mathfrak{SM}_1^\xi$ or $\mathfrak{BS}_\xi$ for any ordinal $0<\xi<\omega_1$. Indeed, in \cite{BCFW}, for every countable ordinal $\xi$, an example was given of a Banach-Saks operator $P_\xi$ from a separable Banach space into itself such that $\textbf{NP}_1(P_\xi)>\xi$.  If $S$ is a separable Banach space such that $P_\xi$ factors through a quotient of a subspace of $S$, then $\textbf{NP}_1(S)>\xi$.  From this it follows that if $S$ is any separable Banach space such that every Banach-Saks operator factors through a quotient of a subspace of $S$, then $S$ contains a copy of $\ell_1$.

\end{remark}

  \section{Relationship between the $\ell_1$ and Szlenk indices}

The factorization result of Theorem \ref{Rosenthal theorem} can be improved for operators mapping into Banach spaces with an unconditional basis. It was shown in \cite{CIllinois} that for any operator $A:X\to Y$, $\textbf{NP}_1(A)\leqslant \omega Sz(A)$ (where we obey the convention that $\omega \infty=\infty$).  It was also shown in \cite{CIllinois} that if $Y$ has an unconditional basis, $Sz(A)\leqslant \textbf{NP}_1(A)$.  It follows that if $Y$ has an unconditional basis and $\textbf{NP}_1(A)\leqslant \omega^\xi$, then $Sz(A)\leqslant \omega^\xi$, and $A$ factors through a Banach space $Z$ with Szlenk index not exceeding $\omega^{\xi+1}$ by \cite{Brooker}. Then $\textbf{NP}_1(Z)\leqslant \omega Sz(Z)\leqslant \omega^{1+\xi+1}$.  It follows that if $\xi$ is infinite and $A:X\to Y$ is an operator into a space with unconditional basis such that $\textbf{NP}_1(A)\leqslant \omega^\xi$, then $A$ factors through a Banach space $Z$ with $\textbf{NP}_1(Z)\leqslant \omega^{\xi+1}$.   We collect this in the following theorem. 

\begin{theorem} For any ordinal $\xi$, if $A:X\to Y$ is a member of $\mathfrak{NP}_1^\xi$ and if $Y$ has an unconditional basis, $A$ factors through a member of $\text{\emph{Space}}(\mathfrak{NP}_1^{1+\xi+1})$. 

\end{theorem}

The assumption of some form of unconditionality  is necessary in order to guarantee that $Sz(A)\leqslant \textbf{NP}_1(A)$.  For example, the James tree $JT$ space fails to be Asplund, so $Sz(JT)=\infty$, while $\textbf{NP}_1(JT)$ is countable, since $JT$ is separable and does not contain an isomorph of $\ell_1$. Moreover, we conclude by presenting a more interesting class of examples demonstrating the lack of a general relationship between the Szlenk and $\textbf{NP}_1$ indices. 

\begin{proposition} There exists a countable ordinal $\gamma$ such that for any ordinal $\xi$, there exists an Asplund space $Z$ such that $\textbf{\emph{NP}}_1(Z)\leqslant \zeta$ and $Sz(Z)>\xi$.    \end{proposition}

\begin{proof}
Given a set $\Lambda$, let $JT_\Lambda$ denote the completion of $c_{00}(\Lambda^{<\nn})$ under the norm $$\|\sum_{t\in \Lambda^{<\nn}} a_t e_t\|=\sup\Bigl\{\bigl(\sum_{i=1}^n |\sum_{t\in \mathfrak{s}_i} a_t|^2\bigr)^{1/2}: (\mathfrak{s}_i)_{i=1}^n \text{\ are disjoint segments}\Bigr\}.$$  Here, a \emph{segment} is a subset of $\Lambda^{<\nn}$ of the form $\{u: s\preceq u\preceq t\}$ for some $s,t\in \Lambda^{<\nn}$.   We let $(e^*_t)_{t\in \Lambda^{<\nn}}$ denote the coordinate functionals on $JT_\Lambda$, noting that these functionals all have norm $1$.  We claim the following facts.  \begin{enumerate}[(i)]\item Any separable subspace $X$ of $JT_\Lambda$ is isometrically isomorphic to a subspace of $JT_\nn$. \item For any set $\Lambda$, $\textbf{NP}_1(JT_\Lambda)\leqslant \textbf{NP}_1(JT_\nn)<\omega_1$. \item If $T$ is a well-founded $B$-tree on $\Lambda$, $JT_\Lambda(T):=[e_t:t\in T]$ is Asplund. \item For any $\xi$, there exists a set $\Lambda$ and a $B$-tree $T$ on $\Lambda$ such that $Sz(JT_\Lambda(T))>\xi$.  \end{enumerate} These facts complete the theorem with $\gamma=\textbf{NP}_1(JT_\nn)$.   We remark that $JT_{\{0,1\}}$ is the usual James tree space defined in \cite{James1} and $JT_\nn$ is the variant of the James tree space defined in \cite{GM}.

(i) By the definition of $JT_\Lambda$, for any $x\in JT_\Lambda$, there exists a countable subset $S(x)$ of $\Lambda^{<\nn}$ such that $x\in [e_t: t\in S(x)]$.  From this it follows that there exists a countable subset $\Lambda(x)$ such that $x\in JT_{\Lambda(x)}\subset JT_\Lambda$.   Hence for any separable subspace $X$ of $JT_\Lambda$, there exists a countable subset $\Lambda_0$ of $\Lambda$ such that $X\subset JT_{\Lambda_0}\subset JT_\Lambda$.    Fix an injection $\phi:\Lambda_0\to \nn$ and define $\varphi:\Lambda_0^{<\nn}\to \nn^{<\nn}$ by $\phi(\varnothing)=\varnothing$ and $\varphi((\lambda_i)_{i=1}^n)=(\phi(\lambda_i))_{i=1}^n$.  Then the operator $\Phi:JT_{\Lambda_0}\to JT_\nn$ which is the linear extension of the function $e_t\mapsto e_{\phi(t)}$ is an isometric embedding of $JT_{\Lambda_0}$ into $JT_\nn$.  

(ii) That $\textbf{NP}_1(JT_\nn)<\omega_1$ follows from the fact that $JT_\nn$ is separable and contains no copy of $\ell_1$.  The fact that for any $\Lambda$, $\textbf{NP}_1(JT_\Lambda)\leqslant \textbf{NP}_1(JT_\nn)$ follows from the fact that if $\textbf{NP}_1(JT_\Lambda)>\textbf{NP}_1(JT_\nn)$, then since $\textbf{NP}_1(JT_\nn)$ is countable, exists a separble subspace $X$ of $JT_\Lambda$ such that $\textbf{NP}_1(X)>\textbf{NP}_1(JT_\nn)$, contradicting (i).

(iii) We prove by induction on $\xi$ that if $T$ is a well-founded $B$-tree on $\Lambda$ with $o(B)\leqslant \xi$, $JT_\Lambda(T)$ is Asplund.  If $\xi=0$, $T=\varnothing$ and $JT_\Lambda(\varnothing)$ is the zero vector space.   Next, assume $T$ is a tree on the set $\Lambda$ with $o(T)=\xi>0$ and $JT_\Lambda(S)$ is Asplund whenever $S$ is a $B$-tree on $\Lambda$ with $o(S)<\xi$.   Let $R$ denote the set of members $\lambda$ of $\Lambda$ such that $(\lambda)\in T$, noting that since $o(T)>0$, $R\neq \varnothing$.   For every $\lambda\in R$, let $T(\lambda)$ denote the set of non-empty sequences $s$ in $\Lambda^{<\nn}$ such $(\lambda)\cat s\in T$.    Then $JT_\Lambda(T)=(\oplus_{\lambda\in R} \text{span}(e_{(\lambda)}\oplus[e_{(\lambda)\cat t}:t\in T(\lambda)])_{\ell_2}$.  Moreover, since $o(T(\lambda))<o(T)$ and for each $\lambda\in R$,  $e_t\mapsto e_{(\lambda)\cat t}$ extends to an isometric isomorphism of $JT_\Lambda(T(\lambda))$ with $ [e_{(\lambda)\cat t}:t \in T(\lambda)]$,  $[e_{(\lambda)\cat t}: t\in T(\lambda)]$ is Asplund.  From this we easily deduce that $JT_\Lambda(T)$ is Asplund.

(iv) Fix an ordinal $\xi$ and let $\Lambda=[0, \xi]\times \nn$.   Let $T$ denote the $B$-tree on $\Lambda$ consiting of all sequences $(\zeta_i, k_i)_{i=1}^n$ such that $n\in \nn$ and $\zeta_1>\ldots >\zeta_n$.   One can easily check by induction that for any ordinal $0\leqslant \zeta\leqslant  \xi$, $T^\zeta$ consists of all sequences $(\zeta_i, k_i)_{i=1}^n$ such that $n\in \nn$ and $\zeta_1>\ldots >\zeta_n\geqslant \zeta$. In particular, $o(T)=\xi+1$.     Moreover, it is easy to see that for every $0\leqslant \zeta\leqslant \xi$, $0<\ee<1$, and $t\in T^\zeta$,  $\sum_{\varnothing\prec s\preceq t} e^*_s|_{JT_\Lambda(T)}$ lies in the $\zeta^{th}$ $\ee$-Szlenk derivation of $B_{JT_\Lambda(T)^*}$, which shows that $Sz(JT_\Lambda(T))>\xi$.    
\end{proof}

\end{document}